\documentclass[11pt]{amsart}
\usepackage{amsmath}             % AMS-LaTeX
\usepackage{amssymb}             % AMS-LaTeX
\usepackage{latexsym}            % LaTeX special symbols

\usepackage{ae}

\numberwithin{figure}{section}

\newtheorem{theorem}{Theorem}[section]
\newtheorem{lemma}[theorem]{Lemma}
\newtheorem{proposition}[theorem]{Proposition}
\newtheorem{corollary}[theorem]{Corollary}
\theoremstyle{definition}
\newtheorem{definition}[theorem]{Definition}

\newtheorem{remark}[theorem]{Remark}
\numberwithin{equation}{section}
\newcommand{\abs}[1]{\left\lvert#1\right\rvert}
\newcommand{\norm}[1]{\left\lVert#1\right\rVert}
\newcommand{\R}{\mathbb{R}}
\newcommand{\Rn}{\R^n}

\newcommand{\dom}{\Omega}
\newcommand{\eps}{\varepsilon}
\newcommand{\dif}{\ensuremath{\,\mathrm{d}}}

\newcommand{\diver}{\operatorname{div}}
\newcommand{\p}{{p(x)}}
\newcommand{\ud}[0]{\ensuremath{\,\textrm{d}}}
\newcommand{\loc}{\textrm{loc}}
\newcommand{\tr}{\operatorname{trace}}
\newcommand{\Om}{\Omega}
\newcommand{\diam}{\operatorname{diam}}
\newcommand{\ol}{\overline}
\DeclareMathOperator*{\essliminf}{ess\,lim\,inf}
\DeclareMathOperator*{\essinf}{ess\,inf}

\def\ocirc#1{\ifmmode\setbox0=\hbox{$#1$}\dimen0=\ht0 \advance\dimen0
  by1pt\rlap{\hbox to\wd0{\hss\raise\dimen0
      \hbox{\hskip.2em$\scriptscriptstyle\circ$}\hss}}#1\else
  {\accent"17 #1}\fi}

\def\vint_#1{\mathchoice%
          {\mathop{\kern 0.2em\vrule width 0.6em height 0.69678ex depth -0.58065ex
                  \kern -0.8em \intop}\nolimits_{\kern -0.4em#1}}%
          {\mathop{\kern 0.1em\vrule width 0.5em height 0.69678ex depth -0.60387ex
                  \kern -0.6em \intop}\nolimits_{#1}}%
          {\mathop{\kern 0.1em\vrule width 0.5em height 0.69678ex depth -0.60387ex
                  \kern -0.6em \intop}\nolimits_{#1}}%
          {\mathop{\kern 0.1em\vrule width 0.5em height 0.69678ex depth -0.60387ex
                  \kern -0.6em \intop}\nolimits_{#1}}}
\newcommand{\limminus}{{\mathchoice{\raise.17ex\hbox{$\scriptstyle -$}}
                {\raise.17ex\hbox{$\scriptstyle -$}}
                {\raise.1ex\hbox{$\scriptscriptstyle -$}}
                {\scriptscriptstyle -}}}
\newcommand{\limplus}{{\mathchoice{\raise.17ex\hbox{$\scriptstyle +$}}
                {\raise.17ex\hbox{$\scriptstyle +$}}
                {\raise.1ex\hbox{$\scriptscriptstyle +$}}
                {\scriptscriptstyle +}}}
\newcommand{\vp}{\varphi}
\newcommand{\parts}[2]{\frac{\partial {#1}}{\partial {#2}}}

\makeatletter
\newcommand{\subjclassname@schuppidoo}{\textup{2010} Mathematics Subject
  Classification}
\makeatother

\begin{document}

\title[Equivalence of viscosity and weak solutions]{Equivalence of
  viscosity and weak solutions for the $\p$-Laplacian}

\author{Petri Juutinen, Teemu Lukkari, and Mikko Parviainen}

\address{Petri Juutinen, Department of Mathematics and Statistics,
P.O.Box 35, FIN-40014 University of Jyv\"askyl\"a, Finland}
\email{petri.juutinen@jyu.fi}

\address{Teemu Lukkari, Department of Mathematical Sciences,
NTNU,
NO-7491 Trondheim,
Norway}
\email{teemu.lukkari@math.ntnu.no}

\address{Mikko Parviainen,
Aalto University, School of Science and Technology, Institute of Mathematics,
PO Box 11000, FI-00076 Aalto,
Finland
}
\email{mikko.parviainen@tkk.fi}

\thanks{}


\subjclass[schuppidoo]{Primary 35J92, Secondary 35D40, 31C45, 35B60}

\date{\today}

\begin{abstract}
  We consider different notions of solutions to the $\p$-Lap\-la\-ce
  equation
  \[
  -\diver(\abs{Du(x)}^{p(x)-2}Du(x))=0
  \]
  with $ 1<\p<\infty$. We show by proving a comparison principle that
  viscosity supersolutions and  $\p$-superharmonic functions of
  nonlinear potential theory coincide. This implies that weak and
  viscosity solutions are the same class of functions, and that
  viscosity solutions to Dirichlet problems are unique. As an
  application, we prove a Rad\'o type removability theorem.
\end{abstract}

\maketitle

\section{Introduction}

During the last fifteen years, variational problems and partial differential
equations with various types of nonstandard growth conditions have become
increasingly popular. This is partly due to their frequent appearance in applications
such as the modeling of electrorheological fluids \cite{AcerbiMingione2,Ruzicka} and
image processing \cite{CLR}, but these problems are very interesting from a purely mathematical
point of view as well.

In this paper, we focus on a particular example, the $\p$-Laplace equation
\begin{equation}\label{intro:pxlap}
  -\Delta_{\p} u(x) :=-\diver(\abs{Du(x)}^{p(x)-2}Du(x))=0
\end{equation}
with $ 1<\p<\infty$. This is a model case of a problem exhibiting so-called $\p$-growth,
which were first considered by Zhikov in \cite{Zhikov}. Our interest is directed
at the very notion of a solution to \eqref{intro:pxlap}.
Since this equation is of divergence form, the
most natural choice is to use the distributional weak solutions, whose
definition is based on integration by parts.  However, if the variable
exponent $x\mapsto \p$ is assumed to be continuously differentiable,
then also the notion of viscosity solutions, defined by means of
pointwise touching test functions, is applicable. Our objective is to
prove that weak and viscosity solutions to the $\p$-Laplace equation
coincide. The proof also implies the uniqueness of viscosity solutions
of the Dirichlet problem. For a constant $p$, similar results were
proved by Juutinen, Lindqvist, and Manfredi in \cite{juutinenlm01}.

The modern theory of viscosity solutions, introduced by Crandall and
Lions in the eighties, has turned out to be indispensable.  It
provides a notion of generalized solutions applicable to fully
nonlinear equations, and crucial tools for results related to
existence, stability, and uniqueness for first and second order
partial differential equations, see for example Crandall, Ishii, and Lions \cite{crandallil92}, Crandall
  \cite{crandall}, and Jensen \cite{jensen88}.  Viscosity theory has also been used in
stochastic control problems, and more recently in stochastic games,
see for example \cite{peresssw09}. The variable
exponent viscosity theory is also useful. Indeed, as an application, we prove a Rad\'o
type removability theorem: if a function $u\in C^1(\Om)$ is a solution
to \eqref{intro:pxlap} outside its set of zeroes $\{x\colon u(x)=0\}$,
then it is a solution in the whole domain $\Om$.  A similar result
also holds for the zero set of the gradient. We do not know how to
prove this result without using both the concept of a weak solution
and that of a viscosity solution.

To prove our main result, we show that viscosity supersolutions of the $\p$-Laplace
equation are the same class of functions as $\p$-superharmonic
functions, defined as lower semicontinuous functions obeying the
comparison principle with respect to weak solutions. The equivalence
for the solutions follows from this fact at once. A simple application
of the comparison principle for weak solutions  shows that $\p$-superharmonic functions are
viscosity supersolutions. The reverse implication, however, requires
considerably more work. To show that a viscosity supersolution obeys
the comparison principle with respect to weak solutions, we first show  that weak solutions of
\eqref{intro:pxlap} can be approximated by the weak solutions of
\begin{equation}
  \label{temp1}
  -\Delta_\p u=-\eps
\end{equation}
and then prove a comparison principle between viscosity supersolutions
and weak solutions of \eqref{temp1}. The comparison principle for
viscosity sub-- and supersolutions can be reduced to this result as
well.

Although the outline of our proof is largely the same as that of \cite{juutinenlm01}
for the constant $p$ case, there are several significant
differences in the details.  Perhaps the most important of them is the
fact that the $\p$-Laplace equation is not translation invariant.
 At first thought, this property may not seem that
 consequential, but one should bear in mind that we are dealing with
 generalized solutions, not with classical solutions.
The core of our argument, the proof of the comparison principle for
viscosity solutions, is based on the maximum principle for
semicontinuous functions.  Applying this principle is not entirely
trivial even for such simple equations as
\begin{equation}
  \label{eq:simple-eq}
  -\diver(a(x)D u)=0,
\end{equation}
with a smooth and strictly positive coefficient $a(x)$. In the case of
the $\p$-Laplacian, the proof is quite delicate. We need to carefully
exploit the information coming from the maximum principle for
semicontinuous functions, and properties such as the local
Lipschitz continuity of the matrix square root as well as the regularity of
weak solutions of \eqref{temp1} play an essential role in the proof.

In addition, we have to take into
account the strong singularity of the equation at the points where
the gradient vanishes and $\p<2$.  Further, since
$1<\p<\infty$, the equations we encounter can be singular in some
parts of the domain and degenerate in others, and we have to find a
way to fit together estimates obtained in the separate cases. For a constant $p$, this problem never
occurs. Finally, if we carefully compute
$\Delta_\p u$, the result is the expression
\begin{equation}\label{intro:expanded}
  \begin{split}
    -\Delta_\p u(x)=&-\abs{D u}^{\p-2}
    \left(\Delta u +(\p-2)\Delta_\infty u\right)\\
    &-\abs{Du}^{\p-2}D\p\cdot Du \,\log\abs{Du},
  \end{split}
\end{equation}
where
\begin{displaymath}
  \Delta_\infty u:=\abs{Du}^{-2} D^2 u \,Du\cdot Du
\end{displaymath}
is the normalized $\infty$-Laplacian.  Obviously, the first order term
involving $\log\abs{Du}$ does not appear if $\p$ is constant.

% At first thought, this property may not seem that
% consequential, but one should bear in mind that we are dealing with
% generalized solutions, not with classical solutions.

The $\p$-Laplacian \eqref{intro:pxlap} is not only interesting in its
own right but also provides a useful test case for generalizing the
viscosity techniques to a wider class of equations as indicated by
\eqref{eq:simple-eq}. One important observation is that our proof
uses heavily the well established theory of weak solutions. More
precisely, we repeatedly exploit the existence, uniqueness and
regularity of the weak solutions to \eqref{intro:pxlap}.  In
particular, we use the fact that weak solutions can be
approximated by the weak solutions of \eqref{temp1}. To emphasize this
point, we consider another variable exponent version of the
$p$-Laplace equation, given by
\begin{equation}\label{intro:normalized_eq}
  -\Delta_\p^N u(x):= -\Delta u(x)-(p(x)-2)\Delta_\infty u(x)=0.
\end{equation}
The study of this equation has been recently set forth in
\cite{hastoa10}, and it certainly looks simpler than
\eqref{intro:expanded}. Indeed, one can quite easily prove a comparison
principle for viscosity subsolutions and strict supersolutions.  However, owing to the
incompleteness of the theory of weak solutions of
\eqref{intro:normalized_eq}, the full comparison principle and the
equivalence of weak and viscosity solutions remain open.

\tableofcontents

\section{The spaces $L^{\p} $ and $W^{1,\p}$}

\label{sec:spaced-out}

In this section, we discuss the variable exponent Lebesgue and Sobolev
spaces. These spaces provide the functional analysis framework for
weak solutions. Most of the results below are from \cite{kovacikr91}.

Let $p\colon \R^n\to [1,\infty)$, called a variable exponent, be in
$C^1(\R^n)$ and let $\Omega$ be a bounded open set in $\R^n$.  We denote
\begin{displaymath}
  p^+=\sup_{x\in\Omega}\p \quad \text{and} \quad p^-=\inf_{x\in \Omega}\p,
\end{displaymath}
and assume that
\begin{equation}
\label{eq:p-bounds}
\begin{split}
 1< p^-\le p^+<\infty.
\end{split}
\end{equation}

Observe that our arguments are local, so it would suffice to assume that
\eqref{eq:p-bounds} holds in compact subsets of $\Om$. However, for simplicity of notation, we use the stronger assumption.

%However, the stronger assumption is natural in the sense that then Corollary \ref{cor:unique} implies that viscosity and weak solutions to the %Dirichlet problem are the same.

The \emph{variable exponent Lebesgue space $L^{p(\cdot)}(\Omega)$}
consists of all measurable functions $u$ defined on $\Omega$ for which
the $\p$-modular
\begin{displaymath}
  \varrho_{\p}(u)=\int_\Omega\abs{u(x)}^{\p}\ud x
\end{displaymath}
is finite.  The Luxemburg norm on this space is defined as
\begin{displaymath}
  \norm{u}_{L^{\p} (\Omega)} =\inf\Big\{\lambda > 0:\int_\Omega\,
  \left\vert \frac{u(x)}{\lambda}\right\vert^{\p}\ud x\leq 1\Big\}.
\end{displaymath}
Equipped with this norm $L^{\p}(\Omega)$ is a Banach space.  If $\p$
is constant, $L^{\p}(\Omega)$ reduces to the standard Lebesgue space.

In estimates, it is often necessary to switch between the norm and the
modular. This is accomplished by the coarse but useful
inequalities
\begin{equation}\label{eq:ModularIneq3}
\begin{split}
  \min\{\norm{u}_{L^{\p} (\Omega)}^{p^+},\norm{u}_{L^{\p}
    (\Omega)}^{p^-}\}&\leq \varrho_{\p}(u) \\
    &\leq
  \max\{\norm{u}_{L^{\p} (\Omega)}^{p^+},\norm{u}_{L^{\p} (\Omega)}^{p^-}\},
\end{split}
\end{equation}
which follow from the definition of the norm in a straightforward manner.
%see \cite[Theorem 2.1]{fanzhao01}.
Note that these inequalities imply the equivalence of convergence in
norm and in modular. More specifically,
\begin{displaymath}
  \norm{u-u_i}_{L^{\p}(\Omega)} \to 0\quad \text{if and only if}\quad
  \varrho_{\p}(u-u_i)\to 0
\end{displaymath}
as $i\to \infty$.

A version of H\"older's inequality,
\begin{equation}\label{eq:hoelder}
  \int_{\Omega}fg\ud x\leq C
  \norm{f}_{L^{\p} (\Omega)}\norm{g}_{L^{p'(x)} (\Omega)},
\end{equation}
holds for functions $f\in L^{\p}(\Omega)$ and $g\in
L^{p'(x)}(\Omega)$, where the conjugate exponent $p'(x)$ of $\p$ is
defined pointwise. Further, if $1< p^-\le p^+<\infty$, the dual of
$L^{\p}(\Omega)$ is $L^{p'(x)}(\Omega)$, and the space
$L^{\p}(\Omega)$ is reflexive.

The \emph{variable exponent Sobolev space $W^{1,\p}(\Omega)$} consists
of functions $u\in L^{p(\cdot)}(\Omega)$ whose weak gradient $D u$
exists and belongs to $L^{p(\cdot)}(\Omega)$.  The space
$W^{1,p(\cdot)}(\Omega)$ is a Banach space with the norm
\begin{displaymath}
      \|u\|_{W^{1,\p}(\Omega)}
      =\|u\|_{L^{\p}(\Omega)}+\|D u\|_{L^{\p} (\Omega)}.
\end{displaymath}
The local Sobolev space $W^{1,\p}_{\loc} (\Omega)$ consists of
functions $u$ that belong to $W^{1,\p}_{\loc} (\Omega')$ for all open
sets $\Omega'$ compactly contained in $\Omega$.

The density of smooth functions in $W^{1,\p}(\Omega)$ turns out to be
a surprisingly nontrivial matter. However, our assumption $\p\in C^1$,
or even the weaker $\log$-H\"older continuity, implies that smooth
functions are dense in the Sobolev space $W^{1,\p}(\Omega)$, see
\cite{Diening, samko02}. Due to the H\"older inequality
\eqref{eq:hoelder}, density allows us to pass from smooth test
functions to Sobolev test functions in the definition of weak
solutions by the usual approximation argument.

The Sobolev space with zero boundary values, $W^{1,\p}_0(\Omega)$, is
defined as the completion of $C_0^\infty(\Omega)$ in the norm of
$W^{1,\p}(\Omega)$. The following variable exponent
Sobolev-Poincar\'e inequality for functions with zero boundary values
was first proven by Edmunds and R{\'a}kosn{\'{\i}}k in
\cite{edmundsr00}.
\begin{theorem}
  \label{thm:poincare}
  For every $u \in W^{1,\p}_0(\Om)$, the inequality
  \begin{equation}
    \label{eq:poincare}
    \norm{u}_{L^{\p}(\Om)}\leq C \diam(\Om) \norm{D u}_{L^{\p}(\Om)}
  \end{equation}
  holds with the constant $C$ depending only on the dimension $n$ and $p$.
\end{theorem}

%%%%%%%%%%%%%%%%%%%%%%%%%%%%%%%%%%%%%%%%%%%%%%%%%%%%%%%%%%%%%%%%%%%%%%%%

\section{Notions of solutions}

\label{sec:notions}

In this section, we discuss the notions of weak solutions,
$\p$-super\-har\-monic functions and viscosity solutions to the
equation
\begin{equation}
  \label{eq:div-equation}
  \begin{split}
    -\Delta_{\p} u(x)= -\diver(\abs{Du(x)}^{p(x)-2}Du(x))=0.
  \end{split}
\end{equation}

\begin{definition}
  \label{def:div-weak}
  A function $u \in W^{1,\p}_{\text{loc}}(\Om)$ is a weak supersolution
  to \eqref{eq:div-equation} in $\Om$ if
  \begin{equation}
    \label{eq:div-weak}
    \begin{split}
      \int_\Om \abs{D u(x)}^{p(x)-2}Du(x)\cdot D\vp(x) \ud x\geq 0
    \end{split}
  \end{equation}
  for every nonnegative test function $\vp \in C^{\infty}_0(\Om)$.
  For subsolutions, the inequality in \eqref{eq:div-weak} is reversed, and a function $u$ is
  a weak solution if it is both a super-- and a subsolution, which
  means that we have equality in \eqref{eq:div-weak} for all $\vp \in
  C^{\infty}_0(\Om)$.
\end{definition}

If $u \in W^{1,\p}(\Om)$, then by the usual approximation argument, we may employ test functions
belonging to $W^{1,\p}_0(\Omega)$. Note also that $u$ is a subsolution if and only if $-u$ is a
supersolution.

Since \eqref{eq:div-equation} is the Euler-Lagrange equation of the
functional
\[
v\mapsto \int_\Om \frac1\p\abs{Dv}^\p\ud x,
\]
the existence of a weak solution $u \in W^{1,\p}(\Omega)$ such that
$u-g\in W^{1,\p}_0(\Om)$ for a given $g\in W^{1,\p}(\Om)$ readily
follows by the direct method of the calculus of variations.  Moreover,
by regularity theory \cite{AcerbiMingione, alkhutov97, FanZhaoQM},
there always exists a locally H\"older continuous representative of a
weak solution.

Assume for a moment that $\p$ is radial, and that $p^+<n$. In such a
case, we may modify a well-known example, the fundamental solution of
the $p$-Laplacian.  Indeed, consider the function
\begin{displaymath}
  v(x)=\int_{\abs{x}}^1(p(r)r^{n-1})r^{-1/(p(r)-1)}\dif r
\end{displaymath}
in the unit ball $B(0,1)$ of $\R^n$. Then $v$ is a solution in
$B(0,1)\setminus \{0\}$, but not a weak supersolution in the whole ball
$B(0,1)$. See \cite[Section 6]{harjulehtohklm07} for the details.  A
computation shows that
\begin{displaymath}
  \int_{B(0,\rho)}\abs{D v}^{p(x)}\dif x=\infty
\end{displaymath}
for any $\rho<1$.  To include functions like $v$ in our discussion, we
use the following class of $\p$-superharmonic functions.

\begin{definition}
  \label{def:div-harmonic}
  A function $u: \Om\to  (-\infty,\infty]$ is \emph{$\p$-superharmonic}, if
  \begin{enumerate}
  \item $u$ is lower semicontinuous
  \item $u$ is finite almost everywhere and
  \item \label{itm:superharm-comp} the comparison principle holds: if
    $h$ is a weak solution to \eqref{eq:div-equation} in
    $D\Subset\Om$, continuous in $\ol D$, and
    \[
    u \geq  h \quad \text{on}\quad \partial D,
    \]
    then
    \[
    u\geq   h  \quad \text{in}\quad D.
    \]
  \end{enumerate}

A function $u: \Om\to  [-\infty,\infty)$ is \emph{$\p$-subharmonic}, if $-u$ is
$\p$-super\-har\-monic.
\end{definition}

In the case of the Laplace equation, that is, $\p\equiv 2$, this potential
theoretic definition of superharmonic functions goes back to
F. Riesz. For constant values $p\ne 2$, $p$--superharmonic functions
and, in particular, their relationship to the weak supersolutions were
studied by Lindqvist in \cite{lindqvist86}.
We next review briefly some relevant facts known in the variable
exponent case. We shall use the lower semicontinuous regularization
\begin{equation}
  \label{essliminf}
  u^*(x)=\essliminf_{y\to x} u(y)=
  \lim_{R \to 0} \essinf_{B(x,R)} u.
\end{equation}

First, every weak supersolution has a lower semicontinuous representative which is $\p$-superharmonic.

\begin{theorem}\label{thm:lscrepresentative}
  Let $u$ be a weak supersolution in $\Om$. Then $u=u^*$ almost everywhere
  and $u^*$ is $\p$-superharmonic.
\end{theorem}

This theorem follows from \cite[Theorem 6.1]{harjulehtohklm07} and
\cite[Theorem 4.1]{HKL}.  For the reverse direction, we have (see
\cite[Corollary 6.6]{harjulehtohklm07})
\begin{theorem}\label{thm:local-bnd}
  A locally bounded $\p$-superharmonic function is a weak
  supersolution.
\end{theorem}

With these results in hand, we may conclude that being a weak solution
is equivalent to being both $\p$-super-- and $\p$-subharmonic. Indeed,
any function with the latter properties is continuous and hence
locally bounded. Then Theorem \ref{thm:local-bnd} implies that the
function belongs to the right Sobolev space, and verifying the weak
formulation is easy. For the converse, it suffices to note that the
comparison principle for the continuous representative of a weak
solution follows from Theorem \ref{thm:lscrepresentative}.

Next we define viscosity solutions of the $\p$-Laplace equation. To
accomplish this, we need to evaluate the operator $\Delta_{\p} $ on
$C^2$ functions. Carrying out the differentiations, we see that
\begin{displaymath}
  \Delta_{\p}\vp(x)=\abs{D \vp}^{p(x)-2}\left(
    \Delta \vp+D p\cdot D \vp\log\abs{D \vp}
    +(p(x)-2)\Delta_{\infty} \vp(x)\right)
\end{displaymath}
for functions $\vp\in C^2(\dom)$, where
\begin{displaymath}
  \Delta_\infty\vp(x)= D^{2}\vp(x) \frac{D\vp(x)}{\abs{D\vp(x)}}\cdot
    \frac{D\vp(x)}{\abs{D\vp(x)}}
\end{displaymath}
is the normalized $\infty$-Laplacian.

In order to use the standard theory of viscosity solutions,
$\Delta_{\p}\vp(x)$ should be continuous in $x$, $D\vp $, and
$D^2\vp$. Since $Dp(x)$ is explicitly involved, it is natural to
assume that $\p\in C^1 $.  However, this still leaves the problem that
at the points were $\p<2$ and $D\vp(x)=0$, the expression
$\Delta_{\p}\vp(x)$ is not well defined. As in \cite{juutinenlm01}, it
will turn out that we can ignore the test functions whose gradient
vanishes at the point of touching.

\begin{definition}
  \label{def:div-viscosity}
  A function $u:\Om\to (-\infty,\infty]$ is a viscosity
  supersolution to \eqref{eq:div-equation}, if
  \begin{enumerate}
  \item $u$ is lower semicontinuous.
  \item $u$ is finite almost everywhere.
  \item \label{itm:visco-def-3} If $\vp \in C^2(\Om)$ is such that $u(x_0) = \vp(x_0)$, $u(x) > \vp(x)$
for $x \neq x_0$, and
    $D\vp(x_0)\neq0$, it
    holds that
    \[
    \begin{split}
      -\Delta_{\p} \vp(x_0)\geq 0.
    \end{split}
    \]
  \end{enumerate}

  A function $u:\Om\to [-\infty,\infty)$ is a viscosity subsolution to
  \eqref{eq:div-equation} if it is upper semicontinuous, finite a.e.,
  and \eqref{itm:visco-def-3} holds with the inequalities reversed.

  Finally, a function is a viscosity solution if it is both a
  viscosity super-- and subsolution.
\end{definition}

We often refer to the third condition above by saying that $\vp$
touches $u$ at $x_0$ from below.
The definition is symmetric in the
same way as before: $u$ is a viscosity subsolution if and only if $-u$
is a viscosity supersolution. Observe that nothing is required from
$u$ at the points in which it is not finite.

One might wonder if omitting entirely the test functions whose
gradient vanishes at the point of touching could allow for ``false''
viscosity solutions for the equation. Our results ensure that this is
not the case. Indeed, we show that the requirements in Definition
\ref{def:div-viscosity} are stringent enough for a comparison
principle to hold between viscosity sub-- and supersolutions, and,
moreover, that the definition is equivalent with the definition of a
weak solution. Finally, we want to emphasize that Definition
\ref{def:div-viscosity} is tailor-made for the equation $-\Delta_\p
v(x)=0$, and it does \emph{not} work as such for example in the case
of a non-homogeneous $\p$-Laplace equation $-\Delta_\p v(x)=f(x)$.

\section{Equivalence of weak solutions and viscosity
  solutions}

\label{main_results}

We turn next to the equivalence between weak and viscosity solutions
to \eqref{eq:div-equation}. This follows from the fact that viscosity
supersolutions and $\p$-superharmonic functions are the same class of
functions.  This is our main result.

\begin{theorem}\label{thm:equiv-supers}
  A function $v$ is a viscosity supersolution to \eqref{eq:div-equation} if and only if it is
  $\p$-superharmonic.
\end{theorem}

As an immediate corollary we have

\begin{corollary}\label{thm:equiv-sols}
  A function $u$ is a weak solution of \eqref{eq:div-equation} if and only if it is a viscosity
  solution of \eqref{eq:div-equation}.
\end{corollary}

Let us now start with the {\em proof of Theorem \ref{thm:equiv-supers}}.
Proving the  fact that a $\p$-superharmonic function is a viscosity supersolution is straightforward, cf.\ for example \cite{manfrediru09}:
Suppose first  that $v$ is $\p$-superharmonic. To see that $v$ is
a viscosity supersolution, assuming the opposite we find a function
$\varphi\in C^2(\Omega)$ such that $v(x_0)=\varphi(x_0)$,
$v(x)>\varphi(x)$ for all $x\not=x_0$, $D\varphi(x_0)\ne 0$, and
\begin{displaymath}
  -\Delta_{\p}\varphi(x_0)<0.
\end{displaymath}
By continuity, there is a radius $r$ such that $D\varphi(x)\ne 0$ and
\begin{displaymath}
  -\Delta_{\p}\varphi(x)<0
\end{displaymath}
for all $x\in B(x_0,r) $. Set
\[
\begin{split}
  m=\inf_{\abs{x-x_0}=r}(v(x)-\varphi(x))>0,
\end{split}
\]
and
\[
\begin{split}
  \widetilde{\varphi}=\varphi+m.
\end{split}
\]
Then $\widetilde{\varphi}$ is a weak subsolution in $B(x_0,r)$, and
$\widetilde{\varphi}\leq v$ on $\partial B(x_0,r)$. Thus
$\widetilde{\varphi}\leq v $ in $B(x_0,r)$ by the comparison principle
for weak sub-- and supersolutions, Lemma \ref{lem:comparison} below,
but
\begin{displaymath}
  \widetilde{\varphi}(x_0)=\varphi(x_0)+m>v(x_0),
\end{displaymath}
which is a contradiction.

The proof of the reverse implication is much more involved.  Let us
suppose that $v$ is a viscosity supersolution. In order to prove that
$v$ is $\p$-superharmonic, we need to show that $v$ obeys the
comparison principle with respect to weak solutions of
\eqref{eq:div-equation}. To this end, let $D\Subset\Om$ and let $h\in
C(\overline D)$ be a weak solution of \eqref{eq:div-equation} such
that $v\ge h$ on $\partial D$. Owing to the lower semicontinuity of $v$,
for every $\delta>0$ there is a smooth
domain $D^\prime\Subset D$ such that $h\le v+\delta$ in $D\setminus D^\prime$. The reason for passing to $D'$ is that we aim to use $h$ as  boundary values and therefore it should belong to the global Sobolev space instead of $W^{1,\p}_\textrm{loc}(D)$.

For $\eps>0$, let $h_{\eps}$ be the
unique weak solution to
\begin{displaymath}
  -\Delta_{\p}h_{\eps}=-\eps, \quad \eps >0
\end{displaymath}
such that $h_{\eps}-h\in W^{1,\p}_0(D^\prime)$.
Then $h_{\eps}$ is locally
Lipschitz in $D^\prime$, see
  \cite{AcerbiMingione, cosciam99}, $v+\delta\geq h_\eps$ on  $\partial D^\prime$ because of the smoothness of $D^\prime$, and it follows from Lemma \ref{lem:conv} below that
$h_{\eps}\to h$ locally uniformly in $D^\prime$ as $\eps\to 0$. Hence, in order
to prove that $v\ge h$ in $D$,  it suffices to prove
that $v+\delta \ge h_{\eps}$ in $D^\prime$ and then let first $\eps\to 0$ and
then $\delta\to 0$. As $v+\delta$ is also a viscosity supersolution of \eqref{eq:div-equation},
the proof of Theorem \ref{thm:equiv-supers} thus reduces to

\begin{proposition}\label{prop:diver_strictcomp}
Let $D^\prime\Subset \Om$, and suppose that $v$ is a viscosity
  supersolution to the $\p$-Laplace equation in $D^\prime$, and let
  $\eps>0$. Assume further that $h_\eps$ is a locally Lipschitz
  continuous weak solution of
  \begin{equation}
  \label{eq:eps-equation2}
  \begin{split}
  -\Delta_{\p} h_\eps=-\eps
  \end{split}
  \end{equation}
   in $D^\prime$ such that
  \begin{displaymath}
    v\geq h_\eps \quad\text{on}\quad  \partial D^\prime.
  \end{displaymath}
  Then
  \begin{displaymath}
      v\geq h_\eps \quad \text{in}\quad  D^\prime.
  \end{displaymath}

  A similar statement holds for viscosity subsolutions $u$, and
  locally Lipschitz continuous weak solutions $\widetilde{h}_\eps$ of
  \begin{equation}
  \label{eq:eps-equation1}
  \begin{split}
  -\Delta_{\p} \widetilde{h}_\eps=\eps.
  \end{split}
  \end{equation}
  In other words, if
  \begin{displaymath}
    u\leq \widetilde{h}_\eps \quad\text{on}\quad  \partial D^\prime,
  \end{displaymath}
  then
  \begin{displaymath}
    u\le \widetilde{h}_\eps \quad \text{in}\quad  D^\prime.
  \end{displaymath}
\end{proposition}

The proof for Proposition \ref{prop:diver_strictcomp} turns out to be
both long and technically complicated.  It requires three Lemmas for
weak solutions that we prove in Section \ref{sec:three_lemmas}
below. The proof itself is given in Section \ref{sec:comparison}.

%We close this section by showing that Proposition
%\ref{prop:diver_strictcomp} in fact implies the comparison principle
%for viscosity sub- and supersolutions of the $\p$-Laplace
%equation.

We close this section by proving a comparison principle for a viscosity subsolution $u$ and a supersolution $v$. Roughly, the idea is to fix a smooth boundary value function which lies between $u$ and $v+\delta, \delta>0,$ at the boundary of an approximating smooth domain and find the unique weak solution $h_\eps$ to
  \eqref{eq:eps-equation2}. Since $h_\eps$ is locally Lipschitz continuous,  we can use Proposition~\ref{prop:diver_strictcomp} for $h_\eps$ and $v$ to show that $v\geq h_\eps$. The argument for the solution $\tilde h_\eps$ to \eqref{eq:eps-equation1} and $u$ is analogous, and Lemma~\ref{lem:conv} completes the proof by  showing that $h_\eps, \tilde h_\eps$ converge to  $h$ as $\eps\to 0$,  where $h$ is the unique weak solution to \eqref{eq:div-equation} with the same boundary values. Details are given below.

\begin{theorem}\label{thm:the-honest-comparison-principle}
  Let $\Omega$ be a bounded domain.  Assume that $u$ is a viscosity
  subsolution, and $v$ a viscosity supersolution such that
  \begin{equation}\label{gen_bndry_cond}
    \limsup_{x\to z}u(x)\leq \liminf_{x\to z}v(z)
  \end{equation}
  for all $z\in\partial \Omega$, where both sides are not
  simultaneously $-\infty$ or $\infty$. Then
  \begin{displaymath}
    u\leq v \quad \text{in}\quad \Omega.
  \end{displaymath}
\end{theorem}

\begin{corollary}
  \label{cor:unique}
  Let $\Omega$ be a bounded domain, and $f:\partial\Omega\to \R$ be a
  continuous function. If $u$ and $v$ are viscosity solutions of
  \eqref{eq:div-equation} in $\Omega$ such that
  \begin{displaymath}
    \lim_{x\to x_0}u(x)=f(x_0)\quad\text{and}\quad \lim_{x\to x_0}v(x)=f(x_0)
  \end{displaymath}
  for all $x_0\in \partial\Omega$, then $u=v$.
\end{corollary}

\begin{proof}[Proof of Theorem \ref{thm:the-honest-comparison-principle}]
  Owing to \eqref{gen_bndry_cond}, for any $\delta>0$ there is a
  smooth subdomain $D\Subset \Omega$ such that
  \begin{displaymath}
    u<v+\delta
  \end{displaymath}
  in $\Omega\setminus D$.  By semicontinuity, there is a smooth
  function $\varphi$ such that
  \begin{displaymath}
    u<\varphi<v+\delta \quad \text{on}\quad  \partial D.
  \end{displaymath}
  Let $h$ be the unique weak solution to \eqref{eq:div-equation} in
  $D$ with boundary values $\varphi$. Then
  \begin{displaymath}
    u< h<v+\delta
  \end{displaymath}
  on $\partial D$, and $h$ is locally Lipschitz continuous in $D$ by
  the local $C^{1,\alpha}$ regularity of $\p$-harmonic functions, see
  \cite{AcerbiMingione, cosciam99}.

  For $\eps>0$, let $h_\eps$ be the unique weak solution to
  \eqref{eq:eps-equation1} such that $h_\eps-h\in W^{1,\p}_0(D)$. Then
  $h_\eps$ is locally Lipschitz in $D$ and it follows from Proposition
  \ref{prop:diver_strictcomp} that $u\leq h_\eps$ in $D$. In view of
  Lemma \ref{lem:conv}, this shows that $u\leq h$ in $D$, and a
  symmetric argument, using equation \eqref{eq:eps-equation2},
gives $h \le v+\delta$ in $D$. Thus we have $u\le
  v+\delta$ in $D$, and since this inequality was already known to
  hold in $\Om\setminus D$, we finally have $u\le v+\delta$ in
  $\Om$. The claim now follows by letting $\delta\to 0$.
\end{proof}

\section{Three Lemmas for weak solutions}\label{sec:three_lemmas}

In this section, we prove three lemmas that are needed in the proofs
of Theorem~\ref{thm:equiv-supers} above, and of
Proposition \ref{prop:diver_strictcomp} in the next section.  The
following well-known vector inequalities will be used several times
below:
\begin{equation}
  \label{eq:mono_est}
  (\abs{\xi}^{q-2}\xi-\abs{\eta}^{q-2}\eta)\cdot(\xi-\eta)\ge
\begin{cases}
2^{2-q}\abs{\xi-\eta}^{q}&\quad\text{if $q\ge 2$,}\\
(q-1)\frac{\abs{\xi-\eta}^2}{(\abs{\xi}+\abs{\eta})^{2-q}}&\quad\text{if $1<q< 2$.}
\end{cases}
\end{equation}
In particular, we have
\begin{equation}
  \label{eq:monotonicity}
  (\abs{\xi}^{p(x)-2}\xi-\abs{\eta}^{p(x)-2}\eta)\cdot(\xi-\eta)>0
\end{equation}
for all $\xi,\eta\in \R^n$ such that $\xi \not= \eta$ and
$1<p(x)<\infty$.

We begin with the following form of the comparison principle. Note
that the second assumption holds if $u$ is a weak subsolution, and $v$
a weak supersolution. For this reason, the lemma is also the basis of
the proof of the $\p$-superharmonicity of weak supersolutions, Theorem
\ref{thm:lscrepresentative}.
\begin{lemma}\label{lem:comparison}
  Let $u$ and $v$ be functions in $W^{1,\p}(\Omega)$ such that
  $(u-v)_+\in W^{1,\p}_0(\Omega)$. If
  \begin{displaymath}
    \int_{\Omega}\abs{D u}^{p(x)-2}D u\cdot D \varphi\dif x\leq
    \int_{\Omega}\abs{D v}^{p(x)-2}D v\cdot D \varphi\dif x
  \end{displaymath}
  for all positive test functions $\varphi\in W^{1,\p}_0(\Omega)$, then $u\leq v$ almost
  everywhere in $\Omega $.
\end{lemma}
\begin{proof}
  By the assumption and \eqref{eq:monotonicity}, we see that
  \[
  \begin{split}
    0\leq \int_\Om (\abs{D u}^{p(x)-2}D u-\abs{D v}^{p(x)-2}Dv)\cdot
    D(u-v)_+\dif x\leq 0.
  \end{split}
  \]
  Thus $D(u-v)_+=0$, and since $(u-v)_+$ has zero boundary values, the
  claim follows.
\end{proof}

\begin{lemma}\label{lem:conv}
  Let $u\in W^{1,\p}(\Omega)$ be a weak solution of
  \begin{displaymath}
    -\Delta_{\p}u=0
  \end{displaymath}
  in $\Omega$, and $u_\eps$ a weak solution of
  \begin{displaymath}
    -\Delta_{\p}u=\eps, \quad \eps>0,
  \end{displaymath}
  such that $u-u_\eps\in W^{1,p(x)}_0(\Omega)$. Then
  \[
  \begin{split}
  u_\eps\to u\quad \text{locally uniformly in}\quad  \Omega,
  \end{split}
  \]
  as $\eps \to 0$.
\end{lemma}
\begin{proof}
  We begin the proof by deriving a very rough estimate for
  $\abs{Du-Du_\eps}$ in $L^{p(x)}$.  To this end, since $u_\eps$
  minimizes the functional
  \[
  v\mapsto \int_\Om \left(\frac1\p\abs{Dv}^\p-\eps v\right)\ud x,
  \]
 using H\"older's inequality, Sobolev-Poincar\'e (Theorem
  \ref{thm:poincare}), and the modular inequalities
  \eqref{eq:ModularIneq3}, we have
  \[
  \begin{split}
    \int_{\Om} \abs{Du_\eps}^\p \ud x \le&\, C \int_\Om \left( \abs{Du}^\p+\eps\abs{u_\eps-u}\right)\ud x \\
    \le&\, C\left( \int_\Om \abs{Du}^\p\ud x+
      \eps\norm{1}_{p'(x)}\norm{u_\eps-u}_{L^\p(\Om)}\right)\\
    \leq&\, C\left(\int_\Om \abs{Du}^\p\ud x+\eps \norm{Du_\eps-Du}_{L^\p(\Om)}\right)\\
    \le&\, C \left( 1+\norm{Du}_{L^\p(\Om)}^{p^+}+\eps\left(\norm{Du}_{L^\p(\Om)}+\norm{Du_\eps}_{L^\p(\Om)}\right)\right).
  \end{split}
  \]
  Alternatively, we could start  by testing the weak formulation for $u_\eps$ with $u-u_\eps\in W^{1,p(x)}_0(\Om)$,
  use Young's inequality, and then continue in the same way as above.
  Using \eqref{eq:ModularIneq3} again, and absorbing one of the terms into the left gives
 \begin{equation}
 \label{eq:preli-rough}
 \begin{split}
  \norm{Du_\eps}_{L^\p(\Om)} \le  C \left( 1+\norm{Du}_{L^\p(\Om)}^{p^+}\right)^{1/p^-}
 \end{split}
 \end{equation}
  with a constant $C$ independent of $\eps$ for all $\eps>0$ small enough, and thus
  \begin{equation}\label{eq:rough}
    \norm{Du-Du_\eps}_{L^\p(\Om)} \le  C \left( 1+\norm{Du}_{L^\p(\Om)}^{p^+}\right).
  \end{equation}

  Next we use $u-u_\eps\in W^{1,p(x)}_0(\Omega)$ as a test-function in
  the weak formulations of $-\Delta_{\p}u=0$ and
  $-\Delta_{\p}u_\eps=\eps$, and subtract the resulting
  equations. This yields
  \begin{equation}\label{eq:eqdiff}
  \begin{split}
\int_{\Om}(\abs{D u}^{p(x)-2}Du-\abs{D  u_\eps}^{p(x)-2}Du_\eps)&\cdot(D u-Du_\eps)\ud x
\\
&=\eps \int_\Om \left(u_\eps-u\right)\ud x.
\end{split}
\end{equation}
The right hand side can be estimated as above:
\[
\eps \int_\Om \left(u_\eps-u\right)\ud x
    \leq C\eps \norm{Du_\eps-Du}_{L^\p(\Om)}.
\]
In order to obtain a suitable lower bound for the left hand side of \eqref{eq:eqdiff}, we use
the two inequalities in \eqref{eq:mono_est}, and therefore we need
to consider separately the subsets $\Om^-:=\{x\in\Om\colon 1<p(x)<2\}$, and
$\Om^+:=\{x\in\Om\colon p(x)\ge 2\}$.

Let us first concentrate on $\Om^-$. By using H\"older's inequality
\eqref{eq:hoelder}, and the modular inequalities
\eqref{eq:ModularIneq3}, we have
 \[
 \label{eq:p-less-2}
 \begin{split}
 \int_{\Om^-}\abs{Du-Du_\eps}^{p(x)} \ud x
  \leq&\, C \norm{\frac{\abs{Du_\eps-Du}^{p(x)}}
{(\abs{D u}+\abs{Du_\eps})^{\frac{p(x)}2(2-p(x))}}}_{L^{\frac2{p(x)}}(\Om^-)}\\
 &\times\norm{(\abs{D u}+\abs{Du_\eps})^{\frac{p(x)}2(2-p(x))}}_{L^{\frac2{2-p(x)}}(\Om^-)}\\
    \leq&\, C\max_{p\in \{{\hat p}^+,{\hat p}^-\}}\left(\int_{\Om^-}
\frac{\abs{Du_\eps-Du}^{2}}{(\abs{D u}+\abs{Du_\eps})^{2-p(x)}} \ud x\right)^{p/2}\\
&\times \left(1+\int_{\Om^-} (\abs{D u}+\abs{Du_\eps})^\p\ud x
\right)^{1/2},
\end{split}
 \]
where ${\hat p}^-=\inf_{\Om^-} p(x) $ and ${\hat p}^+=\sup_{\Om^-} p(x)$.
The vector inequality \eqref{eq:mono_est}, Young's inequality, and the fact that $1<\hat p^-$,  ${\hat p}^+\le 2$ imply
\[
\begin{split}
&\max_{p\in \{{\hat p}^+,{\hat p}^-\}}\left(\int_{\Om^-}
\frac{\abs{Du_\eps-Du}^{2}}{(\abs{D u}+\abs{Du_\eps})^{2-p(x)}} \ud x\right)^{p/2}\\
&\ \le \max_{p\in \{{\hat p}^+,{\hat p}^-\}}
C\left(\int_{\Om^-}(\abs{D u}^{p(x)-2}Du-\abs{D  u_\eps}^{p(x)-2}Du_\eps)\cdot(D u-Du_\eps)\ud x\right)^{p/2}\\
&\ \le C\left(\delta^{\frac2{2-{\hat p}^-}}+\delta^{-\frac2{{\hat p}^-}}\int_{\Om^-}
(\abs{D u}^{p(x)-2}Du-\abs{D  u_\eps}^{p(x)-2}Du_\eps)\cdot(D u-Du_\eps)\ud x\right)
\end{split}
\]
for any $0<\delta<1$, to be chosen later. Combining this with \eqref{eq:preli-rough}, which can be used to bound the term
$(1+\int_{\Om^-} (\abs{D u}+\abs{Du_\eps})^\p\ud x)$,
we obtain
\begin{equation}\label{eq:sing_part}
\begin{split}
&\int_{\Om^-}\abs{Du-Du_\eps}^{\p} \ud x
\le\, C\delta^{\frac2{2-{\hat p}^-}}\\
&\hspace{1 em}+C\delta^{-\frac2{{\hat p}^-}}\int_{\Om^-}
(\abs{D u}^{p(x)-2}Du-\abs{D  u_\eps}^{p(x)-2}Du_\eps)\cdot(D u-Du_\eps)\ud x
\end{split}
\end{equation}
for $0<\delta<1$ and for a constant $C$ depending on $u$ but independent of $\eps$ and $\delta$.

For $\Om^+=\{x\in\Om\colon p(x)\geq 2\}$, \eqref{eq:mono_est} gives
 \[
 %\label{eq:p-larger-2}
 \begin{split}
 \int_{\Om^+} &\abs{Du-Du_\eps}^{p(x)} \ud x \\
 &\leq C \int_{\Om^+}(\abs{D u}^{p(x)-2}Du-\abs{D  u_\eps}^{p(x)-2}Du_\eps)\cdot(D u-Du_\eps)\ud x,
 \end{split}
\]
and summing up this with \eqref{eq:sing_part} and using \eqref{eq:eqdiff}
yields
 \[
 \begin{split}
 \int_{\Om} &\abs{Du-Du_\eps}^{p(x)} \ud x \\
 &\leq C \left( \delta^{\frac2{2-{\hat p}^-}}+\delta^{-\frac2{{\hat p}^-}}\int_{\Om}
(\abs{D u}^{p(x)-2}Du-\abs{D  u_\eps}^{p(x)-2}Du_\eps)\cdot(D u-Du_\eps)\ud x\right)\\
&\le C \left( \delta^{\frac2{2-{\hat p}^-}}+\delta^{-\frac2{{\hat p}^-}} \eps\norm{Du-Du_\eps}_{L^{\p}(\Om)}\right).
 \end{split}
 \]
Now we choose $\delta=\eps^{\frac{(2-{\hat p}^-){\hat p}^-}4}$, and have
\[
\int_{\Om} \abs{Du-Du_\eps}^{p(x)} \ud x
\le C \eps^{{\hat p}^-/2}\Bigg( 1+\norm{Du-Du_\eps}_{L^{\p}(\Om)}\Bigg).
\]
Owing to \eqref{eq:rough} and the modular
inequalities, we obtain
\begin{displaymath}
  \norm{Du-Du_\eps}_{L^{\p}(\Om)}\to 0 \quad \text{as} \quad  \eps\to 0.
\end{displaymath}

It follows from this and another application of Poincar\'e's inequality
as well as inequalities \eqref{eq:ModularIneq3} that
\begin{equation}
  \label{eq:sobo-convergence}
  \begin{split}
   u_\eps\to u  \quad \text{in} \quad W^{1,\p}(\Omega)\quad \text{as}\quad
  \eps\to 0.
  \end{split}
  \end{equation}
  Then we choose $\eps_1$ and $\eps_2$ such that $\eps_1\geq \eps_2$
  and subtract the corresponding equations to get
  \[
  \begin{split}
   \int_\Omega (\abs{D u_{\eps_1}}^{p(x)-2}Du_{\eps_1}-\abs{D  u_{\eps_2}}^{p(x)-2}Du_{\eps_2})
    \cdot D\vp\dif x =(\eps_1-\eps_2)\int_\Omega \vp\dif x\geq 0
  \end{split}
  \]
  for positive $\vp$.  According to Lemma \ref{lem:comparison}, we
  have $u_{\eps_1}\geq u_{\eps_2}$ almost everywhere.  This together
  with \eqref{eq:sobo-convergence} implies that
   \[
  \begin{split}
  u_\eps\to u\quad \text{almost everywhere in}\quad  \Omega.
  \end{split}
  \]
  The claim about the locally uniform convergence follows from  $C^{\alpha}_{loc}$-estimates for $u_\eps$ which are
  uniform in $\varepsilon$ due to the results in \cite[Section 4]{FanZhaoQM}.
\end{proof}

We use the next lemma to deal with the singularity in the equation in
the region where $1<p(x)<2$.
\begin{lemma}\label{lem:X.X}
  Let $v_\eps\in W^{1,\p}(\Om)$ be a weak solution of
  \begin{equation}\label{eq:X.X}
    -\Delta_{\p}v=\eps.
  \end{equation}
  Suppose that $\vp\in C^2(\Om)$ is
 such that $v_\eps(x_0) = \vp(x_0)$, $v_\eps(x) > \vp(x)$
for $x \neq x_0$, and that either $x_0$
  is an isolated critical point of $\vp$, or $D\vp(x_0)\not=0$. Then
  \begin{displaymath}
    \limsup_{\substack{x\to x_0\\x\not=x_0}}(-\Delta_{\p}\vp(x))\geq \eps.
  \end{displaymath}
\end{lemma}
\begin{proof}
  We may, of course, assume that $x_0=0$. We make a counterassumption,
  which yields a radius $r>0$ such that
  \begin{displaymath}
    D\vp(x)\not=0 \quad \text{and} \quad -\Delta_{\p}\vp(x)<\eps
  \end{displaymath}
  for $0<\abs{x}<r $.

  We aim at showing first that $\vp$ is a weak subsolution of
  \eqref{eq:X.X} in $B_r=B(0,r)$. Let $0<\rho<r$. For any positive
  $\eta\in C_0^\infty(B_r)$, we have
  \begin{align*}
    -\int_{\abs{x}=\rho}\eta\abs{D\vp}^{p(x)-2}D\vp
    \cdot\frac{x}{\rho}\dif S
    =& \int_{\rho<\abs{x}<r}\abs{D\vp}^{p(x)-2}D\vp\cdot D\eta\dif x\\
    & +\int_{\rho<\abs{x}<r}(\Delta_{\p}\vp)\eta\dif x
  \end{align*}
  by the divergence theorem. The left hand side tends to zero as
  $\rho\to 0$, since
  \begin{displaymath}
    \abs{\int_{\abs{x}=\rho}\eta\abs{D\vp}^{p(x)-2}D\vp
      \cdot\frac{x}{\rho}\dif S}\leq
    C\norm{\eta}_{\infty}\max\{\norm{D \vp}_\infty^{p^+-1},
    \norm{D \vp}_\infty^{p^--1}\}\rho^{n-1}.
  \end{displaymath}
  By the counterassumption,
  \begin{displaymath}
    \int_{\rho<\abs{x}<r}\eta\Delta_{\p}\vp\dif x\geq
    -\eps\int_{\rho<\abs{x}<r}\eta\dif x\geq -\eps\int_{B_r}\eta\dif
    x.
  \end{displaymath}
  Letting $\rho$ tend to zero, we see that
  \begin{displaymath}
    \int_{B_r}\abs{D\vp}^{p(x)-2}D\vp\cdot D\eta\dif x\leq
    \eps\int_{B_r}\eta\dif x,
  \end{displaymath}
  which means that $\vp$ is indeed a weak subsolution.

  Now a contradiction follows from the comparison principle in a
  similar fashion as in the first part of the proof of Theorem
  \ref{thm:equiv-supers}. Indeed, we have $m=\inf_{\partial
    B_r}(v_\eps-\vp)>0$. Then $\widetilde{\vp}=\vp+m$ is a weak
  subsolution such that $\widetilde{\vp}\leq v_\eps$ on $\partial
  B_r$, but $\widetilde{\vp}(0)>v_\eps(0)$.
\end{proof}

\section{The comparison principle}\label{sec:comparison}

As seen in Section \ref{main_results}, Proposition
\ref{prop:diver_strictcomp} is the core of the proof of the equivalence of weak and viscosity solutions.
To prepare for its proof, we write
$\Delta_{\p}\vp(x) $ in a more convenient form.  For a vector $\xi\ne
0$, $\xi\otimes \xi$ is the matrix with entries $\xi_i\xi_j$. Let
\[
A(x,\xi)=\abs{\xi}^{p(x)-2}\left(I+(p(x)-2)\frac{\xi}{\abs{\xi}}\otimes\frac{\xi}{\abs{\xi}}\right),
\]
\[
B(x,\xi)=\abs{\xi}^{p(x)-2}\log\abs{\xi}\xi\cdot D p(x),
\]
and
\begin{displaymath}
  F(x,\xi,X)=\tr(A(x,\xi)X)+B(x,\xi)
\end{displaymath}
for $x\in\Omega $, $\xi\in\R^n$, and $X$ a symmetric $n\times n$ matrix.
Then we may write
\begin{align*}
  \Delta_{\p}\vp(x)=&F(x,D\vp(x),D^2\vp(x))\\
  =&\tr(A(x,D\vp(x))D^2\vp(x))+B(x,D\vp(x))
\end{align*}
if $D\vp(x)\ne 0$.

We prove the claim about viscosity subsolutions in Proposition
\ref{prop:diver_strictcomp}, the case of supersolutions then following
by symmetry.  For the convenience of the reader, we repeat
the statement before proceeding with the proof.

\begin{proposition}\label{prop:diver_strictcomp_restated}
  Let $\dom$ be a bounded domain, and suppose that $u$ is a viscosity
  subsolution to the $\p$-Laplace equation, and $v$ is a locally
  Lipschitz continuous weak solution of
  \begin{equation}
  \label{eq:eps-equation}
  \begin{split}
  -\Delta_{\p} v=\eps,\quad \eps>0,
  \end{split}
  \end{equation}
   in $\dom$ such that
  \begin{displaymath}
    u\leq v \text{ on }  \partial\dom.
  \end{displaymath}
  Then
  \begin{displaymath}
      u\le v \text{ in } \dom.
  \end{displaymath}
\end{proposition}

\begin{proof}
  The argument follows the usual outline of proving a comparison
  principle for viscosity solutions to second-order elliptic
  equations.  We argue by contradiction and assume that $u-v$ has a
  strict interior maximum, that is,
  \begin{equation}\label{antithesis}
    \sup_\dom (u-v)> \sup_{\partial\dom}(u-v).
  \end{equation}
  We proceed by doubling the variables; consider the functions
  $$
  w_j(x,y)=u(x)- v(y)-\Psi_j(x,y), \quad j=1,2,\dots,
  $$
  where
  $$
  \Psi_j(x,y)=\tfrac jq|x-y|^q,
  $$
  and
  \[
  q>\max\{2,\frac{p^-}{p^- -1}\},\qquad p^-=\inf_{x\in\Omega} p(x) > 1.
  \]
  Let $(x_j,y_j)$ be a maximum of $w_j$ relative to
  $\overline\dom\times\overline\dom$.  By \eqref{antithesis}, we see
  that for $j$ sufficiently large, $(x_j,y_j)$ is an interior
  point. Moreover, up to selecting a subsequence, $x_j\to \hat x$ and
  $y_j\to \hat x$ as $j\to\infty$ and $\hat x$ is a maximum point for
  $u-v$ in $\dom$.  Finally, since
  \[
  u(x_j)-v(x_j) \le u(x_j)-v(y_j)-\frac{j}q\abs{x_j-y_j}^q,
  \]
  and $v$ is locally Lipschitz, we have
  \[
  \frac{j}q\abs{x_j-y_j}^q \le v(x_j)-v(y_j) \le C \abs{x_j-y_j},
  \]
  and hence dividing by $\abs{x_j-y_j}^{1-\delta}$ we get
  \begin{equation}\label{improved-rate}
    j\abs{x_j-y_j}^{q-1+\delta}\to 0\qquad\text{as $j\to\infty$ for any $\delta>0$}.
  \end{equation}
Observe that although $u$ is, in general, an extended real valued function,
it follows from Definition \ref{def:div-viscosity} that $u$ is finite at $x_j$.

  In what follows, we will need the fact that $x_j\not=y_j$. To see
  that this holds, let us denote
  $$
  \vp_j(y)= -\Psi_j(x_j,y) + v(y_j)+\Psi_j(x_j,y_j),
  $$
  and observe that since
  $$
  u(x)- v(y)-\Psi_j(x,y) \le u(x_j)- v(y_j)-\Psi_j(x_j,y_j)
  $$
  for all $x,y\in\dom$, we obtain by choosing $x=x_j$ that
  $$
  v(y)\ge -\Psi_j(x_j,y) + v(y_j)+\Psi_j(x_j,y_j)
  $$
  for all $y\in\dom$. That is, $\vp_j$ touches $v$ at $y_j$ from below, and thus
  \begin{equation}\label{eps-ineq}
  \limsup_{y\to y_j} (-\Delta_{\p}\vp_j(y))\ge \eps
  \end{equation}
  by Lemma \ref{lem:X.X}. On the other hand, a calculation yields
  \[
  \begin{split}
  \Delta_{\p} \vp_j(y) =&\, j^{p(y)-1}\abs{x_j-y}^{(q-1)(p(y)-2)+q-2}\Big[
  n+q-2+(p(y)-2)(q-1)\\
  &\,+\log(j \abs{x_j-y}^{q-1})(x_j-y)\cdot Dp(y),
  \Big]
  \end{split}
  \]
where
\[
(q-1)(p(y)-2)+q-2=q(p(y)-1)-p(y)>0
\]
by the choice of $q$. Hence if $x_j=y_j$, we would have
\[
\limsup_{y\to y_j} (-\Delta_{\p}\vp_j(y))=0,
\]
contradicting \eqref{eps-ineq}. Thus $x_j\ne y_j$ as desired.

  For equations that are continuous in all the variables, viscosity
  solutions may be equivalently defined in terms of the closures of
  super-- and subjets. The next aim is to exploit this fact, together
  with the maximum principle for semicontinuous functions, see
  \cite{crandall, crandallil92, koike04}.  Since $(x_j,y_j)$ is a
  local maximum point of $w_j(x,y)$, the aforementioned principle
  implies that there exist symmetric $n\times n$ matrices $X_j,Y_j$
  such that
  \begin{align*}
    (D_x\Psi_j(x_j,y_j),X_j)\in&\overline J^{2,+} u(x_j),\\
    (-D_y\Psi_j(x_j,y_j),Y_j)\in& \overline J^{2,-} v(y_j),
  \end{align*}
  where $\overline J^{2,+} u(x_j)$ and $\overline J^{2,-} v(y_j)$ are
  the closures of the second order superjet of $u$ at $x_{j}$ and the
  second order subjet of $v$ at $y_{j}$, respectively.
  Further, writing $z_j=x_j-y_j$, the matrices $X_j$ and $Y_j$ satisfy
  \begin{equation}\label{matineq}
    \begin{split}
      \left(\begin{array}{@{}cc@{}}
          X_j&0\\
          0&-Y_j\end{array}\right)
      \le&\, D^2\Psi_j(x_j,y_j)+\frac 1j
      \left[D^2\Psi_j(x_j,y_j)
      \right]^2\\
      =&\, j(|z_j|^{q-2}+2|z_j|^{2q-4})\left(\begin{array}{@{}cc@{}}
          I&-I\\-I&I\end{array}\right)\\
      +& j(q-2)(|z_j|^{q-4}+2q|z_j|^{2q-6})\left(\begin{array}{@{}cc@{}}
          z_j\otimes z_j&-z_j\otimes z_j\\
          -z_j\otimes z_j& z_j\otimes z_j\end{array}\right),
    \end{split}
  \end{equation}
  where
  \begin{align*}
    D^2\Psi_j(x_j,y_j)=&
    \begin{pmatrix}
       D_{xx} \Psi_j(x_j,y_j)&D_{xy} \Psi_j(x_j,y_j)\\
       D_{yx} \Psi_j(x_j,y_j)&D_{xx} \Psi_j(x_j,y_j)
    \end{pmatrix}\\
    =&j\abs{z_j}^{q-2}
    \begin{pmatrix}
      I&-I\\
      -I&I
    \end{pmatrix}\\
    & +j(q-2)\abs{z_j}^{q-4}
    \begin{pmatrix}
       z_j\otimes z_j&-z_j\otimes z_j\\
       -z_j\otimes z_j& z_j\otimes z_j
    \end{pmatrix}
  \end{align*}
  Observe that \eqref{matineq} implies
  \begin{equation}
    \label{eq:sharp-max-princ-estimate}
    \begin{split}
     X_j\xi\cdot\xi - Y_j\zeta\cdot\zeta
     &\le\, j\left[|z_j|^{q-2}+2|z_j|^{2q-4}\right]\abs{\xi-\zeta}^2\\
    &\phantom{\le\,}+j(q-2)\left[|z_j|^{q-4}+2q|z_j|^{2q-6}\right] (z_j\cdot (\xi-\zeta))^2\\
    &\le\, j\left[(q-1)|z_j|^{q-2}+2(q-1)^2|z_j|^{2(q-2)}\right]|\xi-\zeta|^2
  \end{split}
  \end{equation}
  for all $\xi,\zeta\in\Rn$.

  Now, $u$ is a viscosity subsolution of the $\p$-Laplace equation,
  and $v$ a viscosity solution of $-\Delta_{\p} v=\eps$.  By the
  equivalent definition in terms of jets, we obtain that
  $$
  -\tr(A(x_j,\eta_j)X_j)-B(x_j,\eta_j) \le 0
  $$
  and
  $$
  -\tr(A(y_j,\eta_j)Y_j)-B(y_j,\eta_j) \ge \eps.
  $$
  Here it is crucial that
  \[
  \begin{split}
  \eta_j=D_x\Psi_j(x_j,y_j)=-D_y\Psi_j(x_j,y_j)=j\abs{x_j-y_j}^{q-2}(x_j-y_j)
  \end{split}
  \]
  is nonzero as observed above. This guarantees that the
  $\p$-Laplace equation is non-singular at the neighborhoods of
  $(x_j,\eta_j,X_j)$ and $(y_j,\eta_j,Y_j)$, which in turn allows us
  to use jets. Notice also that since $v$ is locally Lipschitz, there
  is a constant $C>0$ such that $\abs{\eta_j}\le C$ for at least large
  $j$'s, for the reason that $(\eta_j,Y_j)\in\overline J^{2,-}
  v(y_j)$.

  Since $\eta_j\not=0$, $A(\cdot,\cdot)$ is positive definite, so that
  its matrix square root exists.  We denote
  $A^{1/2}(x_j)=A(x_j,\eta_j)^{1/2}$ and
  $A^{1/2}(y_j)=A(y_j,\eta_j)^{1/2}$.  Observe that the matrices
  $X_j$, $Y_j$ as well as $A(\cdot,\cdot)$, and $A^{1/2}(\cdot)$ are
  symmetric. We use matrix calculus to obtain
  \[
  \begin{split}
  \tr(A(x_j,\eta_j)X_j )&= \tr(A^{1/2}(x_j)A^{1/2}(x_j) X_j)\\
  &=\tr(A^{1/2}(x_j)^T X_j A^{1/2}(x_j))\\
  &=\sum_{k=1}^n X_j A_k^{1/2}(x_j)\cdot A_k^{1/2}(x_j),
  \end{split}
  \]
  where $A_k^{1/2}(x_j)$ denotes the $k$th column of $A^{1/2}(x_j)$
  This together with \eqref{eq:sharp-max-princ-estimate} implies
  \[
  \begin{split}
    0<\eps \le&\, B(x_j,\eta_j)-B(y_j,\eta_j)\\
    &+\sum_{k=1}^n X_j A_k^{1/2}(x_j)\cdot A_k^{1/2}(x_j)
    - \sum_{k=1}^n Y_j A_k^{1/2}(y_j)\cdot A_k^{1/2}(y_j)\\
    \le&\, B(x_j,\eta_j)-B(y_j,\eta_j)
+Cj\abs{x_j-y_j}^{q-2}\norm{A^{1/2}(x_j)-A^{1/2}(y_j)}_2^2\\
    \le&\, B(x_j,\eta_j)-B(y_j,\eta_j)\\
    &+\frac{Cj\abs{x_j-y_j}^{q-2}}{
      \Big(\lambda_{\min}(A^{1/2}(x_j))+\lambda_{\min}(A^{1/2}(y_j))\Big)^2
    }
    \norm{A(x_j,\eta_j)-A(y_j,\eta_j)}_2^2.
  \end{split}
  \]
  The last inequality is the local Lipschitz continuity of $A\mapsto
  A^{1/2}$, see \cite[p. 410]{hornj85}, and $\lambda_{\min}(M)$ denotes
the smallest eigenvalue of a symmetric $n\times n$ matrix $M$.

  Since $\p \in C^1(\R^n)$, and
  \[
  \begin{split}
  \abs{\abs{\eta_j}^{p(x_j)-2}-\abs{\eta_j}^{p(y_j)-2}}&=\abs{\exp(\log(\abs{\eta_j}^{p(x_j)-2}))
-\exp(\log(\abs{\eta_j}^{p(y_j)-2}))}\\
  &\leq \abs{\parts{\,\exp((s-2)
  \log\abs{\eta_j})}{s}}\abs{p(x_j)-p(y_j)}\\
  &= \abs{\log(\abs{\eta_j})}\abs{\eta_j}^{s-2} \abs{p(x_j)-p(y_j)},
  \end{split}
  \]
  for some $s\in [p(x_j),p(y_j)]$,
  we have
  \[
  \begin{split}
    B(x_j,\eta_j)&-B(y_j,\eta_j)\\
    &=\,
    \abs{\eta_j}^{p(x_j)-2}\log\abs{\eta_j}\eta_j\cdot D p(x_j)-
    \abs{\eta_j}^{p(y_j)-2}\log\abs{\eta_j}\eta_j\cdot D p(y_j)\\
    &\le\, \abs{\eta_j}^{p(x_j)-1}\abs{\log\abs{\eta_j}}\,\abs{D p(x_j)-D p(y_j)}\\
    &\hspace{1 em}+\abs{\eta_j}\abs{\log\abs{\eta_j}}\,\abs{D p(y_j)}
    \,\abs{\abs{\eta_j}^{p(x_j)-2}-\abs{\eta_j}^{p(y_j)-2}}\\
    &\le\, \abs{\eta_j}^{p(x_j)-1}\abs{\log\abs{\eta_j}}\, \abs{D p(x_j)-D p(y_j)}\\
    &\hspace{1 em}+C\abs{\eta_j}^{s-1}\log^2\abs{\eta_j} \abs{p(x_j)-p(y_j)}.
  \end{split}
  \]
  Moreover,
  \begin{multline*}
    \norm{A(x_j,\eta_j)-A(y_j,\eta_j)}_2\\
    \begin{aligned}
      \le&\, \abs{\abs{\eta_j}^{p(x_j)-2}-\abs{\eta_j}^{p(y_j)-2}}\\
      &+\abs{\abs{\eta_j}^{p(x_j)-2}(p(x_j)-2)-\abs{\eta_j}^{p(y_j)-2}(p(y_j)-2)}\\
      \le&\, \max\Big\{3-p(x_j),p(x_j)-1\Big\}\abs{\log\abs{\eta_j}}\,
      \abs{\eta_j}^{s-2}\abs{p(x_j)-p(y_j)}\\
      &+\abs{\eta_j}^{p(y_j)-2}\abs{p(x_j)-p(y_j)}\\
      \le&\, C\left( (p^++1)\abs{\log\abs{\eta_j}}\,\abs{\eta_j}^{s-2}
        +\abs{\eta_j}^{p(y_j)-2}\right)\abs{x_j-y_j},
    \end{aligned}
  \end{multline*}
  and
  \[
  \begin{split}
    \lambda_{\min}(A^{1/2}(x_j))=&\,\left(\lambda_{\min}(A(x_j,\eta_j))\right)^{1/2}
    = \left(\min_{\abs{\xi}=1} A(x_j,\eta_j)\xi\cdot\xi\right)^{1/2}\\
    \ge&\, \min\left\{1,\sqrt{p(x_j)-1}\right\}\abs{\eta_j}^{\frac{p(x_j)-2}2}.
  \end{split}
  \]
  Thus
  \begin{equation}\label{final_est}
    \begin{split}
      0<\eps &\le\,  \abs{\eta_j}^{p(x_j)-1}\abs{\log\abs{\eta_j}}\, \abs{D p(x_j)-D p(y_j)}\\
    &\hspace{1 em}+C\abs{\eta_j}^{s-1}\log^2\abs{\eta_j} \abs{p(x_j)-p(y_j)}\\
      &\hspace{1 em}+\frac{C\left((p^++1)\abs{\log\abs{\eta_j}}\,\abs{\eta_j}^{s-2}
          +\abs{\eta_j}^{p(y_j)-2}\right)^2}
      {\min\{1,p^--1\}
\left(\abs{\eta_j}^{\frac{p(x_j)-2}2}+\abs{\eta_j}^{\frac{p(y_j)-2}2}\right)^2}
      j\abs{x_j-y_j}^q.
    \end{split}
  \end{equation}
  The first two terms on the right hand are easily shown to converge
  to 0 as $j\to\infty$. Indeed, since $x_j\to \hat x$, $p(\hat x)>1$
  and $\abs{\eta_j}\le C$, we have that
  $\abs{\eta_j}^{p(x_j)-1}\abs{\log\abs{\eta_j}}$ remains bounded as
  $j\to\infty$. Thus, owing to the continuity of $x\mapsto Dp(x)$, the
  first term converges to zero as $j\to\infty$. The second term is
  treated in a similar way.

  Finally, we deal with the third term. Recalling that $\abs{\eta_j}=
  j\abs{x_j-y_j}^{q-1}$, we have
  \[
  \begin{split}
    &\left(
      \frac{\abs{\log\abs{\eta_j}}\abs{\eta_j}^{s-2}}{\abs{\eta_j}^{\frac{p(x_j)-2}2}
        +\abs{\eta_j}^{\frac{p(y_j)-2}2}}
    \right)^2 j\abs{x_j-y_j}^q
    \le\, \log^2\abs{\eta_j} \abs{\eta_j}^{2s-p(x_j)-2}j\abs{x_j-y_j}^q\\
    &\quad\le\, \log^2\abs{\eta_j} j^{2s-p(x_j)-1}\abs{x_j-y_j}^{q+(q-1)(2s-p(x_j)-2)}\\
    &\quad=\, \log^2(j\abs{x_j-y_j}^{q-1}) \Big[j\abs{x_j-y_j}^{q-1} \Big]^{2s-p(x_j)-1}\abs{x_j-y_j}\\
    &\quad=\, \log^2(j\abs{x_j-y_j}^{q-1}) \Big[j\abs{x_j-y_j}^{q-1+\delta} \Big]^{2s-p(x_j)-1}
    \abs{x_j-y_j}^{1-\delta(2s-p(x_j)-1)}.
  \end{split}
  \]
  Now, since $2s-p(x_j)-1\to p(\hat x)-1>0$ and
  $j\abs{x_j-y_j}^{q-1+\delta}\to 0$ as $j\to\infty$, we see that
  \[
  \Big[j\abs{x_j-y_j}^{q-1+\delta} \Big]^{2s-p(x_j)-1}\to 0
  \qquad\text{as $j\to\infty$}.
  \]
  Further, we write
  \begin{multline*}
    \log^2(j\abs{x_j-y_j}^{q-1})\abs{x_j-y_j}^{1-\delta(2s-p(x_j)-1)}\\
    =\log^2(j\abs{x_j-y_j}^{q-1})(j\abs{x_j-y_j}^{q-1})^{
      \frac{1-\delta(2s-p(x_j)-1)}{q-1}} \left(\frac1j\right)^{\frac{(1-\delta(2s-p(x_j)-1))}{q-1}},
  \end{multline*}
  and note that choosing $\delta>0$ so small that $1-\delta(p^+-1)>0$
  suffices for the third term to converge to zero as $j\to\infty$. A
  contradiction has been reached.
\end{proof}

\section{An application: a Rad\'o type removability theorem}

\label{sec:rado}

The classical theorem of Rad\'o says that, if the continuous complex
function $f$ is analytic when $f(z)\ne 0$, then it is analytic in its
domain of definition. This result has been extended for solutions of
various partial differential equations, including the Laplace equation
and the $p$-Laplace equation, see the references in \cite{juutinenl05}.
Here we prove a corresponding
removability result for $\p$-harmonic functions.

\begin{theorem}\label{thm:rado}
  If a function $u\in C^1(\Omega)$ is a weak solution of \eqref{eq:div-equation} in
  $\Omega\setminus\{x\colon u(x)=0\}$, then $u$ is a weak solution in
  the entire domain $\Omega$.
\end{theorem}

 \begin{proof}
 A key step in the proof is to observe that if $u\in C^1(\Omega)$ is a weak solution in
   \[
   \begin{split}
   \Om\setminus\{x\colon \,u(x)=0 \},
   \end{split}
   \]
  then it is a weak solution in $\Om\setminus U$, where
  \[
   \begin{split}
   U:=\{x\colon \,u(x)=0  \textrm{ and } Du(x)\neq 0\}.
   \end{split}
   \]
  This readily follows from Corollary \ref{thm:equiv-sols}, because in the definition of a viscosity solution we
ignore the test functions with $D\vp(x_0)=0$. Since $u\in C^1(\Omega)$, it follows that if $Du(x_0)=0$, then $D\vp(x_0)=0$.

Thus the original problem has been reduced to proving the removability
of $U$, which is locally a $C^1$-hypersurface.  There are at least two
ways to accomplish this. One option is to apply \cite[Theorem
2.2]{juutinenl05}, which means using viscosity solutions and an
argument similar to Hopf's maximum principle. The second alternative is
to use a coordinate transformation and map $U$ to a hyperplane, and
then prove the removability of a hyperplane by a relatively simple
computation. The price one has to pay in this approach is that the
equation changes, but fortunately this is allowed in \cite[Lemma
2.22]{martio88}.
\end{proof}

\begin{remark}
We do not know how to prove Theorem \ref{thm:rado} without using viscosity solutions, not even in the
simpler case when $\p$ is constant. In particular, the removability of a level set is an open question
for the weak solutions of \eqref{eq:div-equation} when $\p$ is, say, only continuous. Theorem \ref{thm:rado}
fails if $u$ is assumed to be only Lipschitz continuous.
\end{remark}

\section{The normalized $\p$-Laplacian}\label{sec:normalized}

There has recently been some interest in another variable exponent
version of the $p$-Laplacian, called the normalized $\p$-Laplacian
\begin{equation}
  \label{eq:normalized-equation}
  \begin{split}
    \Delta^N_{p(x)} u(x):= \Delta u(x)+(p(x)-2)\Delta_\infty u(x),
  \end{split}
\end{equation}
where
\[
\Delta_\infty u(x)= D^2 u(x)\frac{Du(x)}{\abs{Du(x)}}\cdot\frac{Du(x)}{\abs{Du(x)}}.
\]
In particular, the weak solutions of the equation $-\Delta^N_{p(x)} u(x)=0$, or rather a scaled version of it,
were studied by Adamowicz and Hästö \cite{hastoa10} in connection with mappings of finite
distortion. In this section, we prove the comparison principle for viscosity subsolutions and strict viscosity
supersolutions of this equation. Hence we obtain an almost exact analogue of Proposition \ref{prop:diver_strictcomp}.
However, owing to the fact that some of the major tools in the weak theory (comparison principle, uniqueness, stability)
for the normalized $\p$-Laplace equation are still missing,
the uniqueness of viscosity solutions and the equivalence of weak and viscosity solutions remain open.

The normalized $\p$-Laplacian has a bounded singularity at the points where the gradient $Du$ vanishes, and the
viscosity solutions can thus be defined in a standard way, by using the upper and lower semicontinuous
envelopes of the operator \eqref{eq:normalized-equation}. For any constant $c\in\R$, we say that an upper semicontinuous function
$v\colon\Om\to\R$ is a viscosity subsolution to $-\Delta^N_{p(x)} v(x)=c$ if, whenever
$\vp \in C^2(\Om)$ is such that $v(x_0) = \vp(x_0)$, $v(x) < \vp(x)$ for $x \neq x_0$, then
\[
\begin{cases}
\phantom{-\Delta \vp(x_0)-(p(x_0)-2)}-\Delta^N_{p(x)} \vp(x_0)\le c,& \text{if $D\vp(x_0)\ne 0$}\\
-\Delta \vp(x_0)-(p(x_0)-2)\lambda_{\min}(D^2\vp(x_0))\le c,& \text{if $D\vp(x_0)= 0$ and $p(x_0)< 2$},\\
-\Delta \vp(x_0)-(p(x_0)-2)\lambda_{\max}(D^2\vp(x_0))\le c,& \text{if $D\vp(x_0)= 0$ and $p(x_0)\ge 2$},
\end{cases}
\]
where $\lambda_{\max}(A)$ and $\lambda_{\min}(A)$ denote the largest and the smallest eigenvalue, respectively,
of a symmetric $n\times n$ matrix $A$. And as usual, a lower semicontinuous function $w$ is a viscosity supersolution
if $-w$ is a viscosity subsolution.

\begin{proposition}\label{norm_strictcomp}
  Suppose that $\dom$ is a bounded domain, $u\in C(\overline\dom)$ is
  a viscosity subsolution to $-\Delta_{\p}^Nu=0$, and $v\in
  C(\overline\dom)$ is a viscosity supersolution of
  \begin{displaymath}
    -\Delta^N_{\p} v=\eps,\quad \eps>0
  \end{displaymath}
  in $\dom$ such that
  \begin{displaymath}
    u\leq v \quad \text{on}\quad \partial\dom.
  \end{displaymath}
  Then
  \begin{displaymath}
      u\le v \quad\text{in}\quad \dom.
  \end{displaymath}
\end{proposition}

\begin{proof}
  The proof is quite similar to that of Proposition
  \ref{prop:diver_strictcomp}.  We again argue by contradiction and
  assume that $u-v$ has an interior maximum.  Consider the functions
  $$
  w_j(x,y)=u(x)- v(y)-\Psi_j(x,y), \quad j=1,2,\dots,
  $$
  where this time
  $$
  \Psi_j(x,y)=\tfrac j4|x-y|^4,
  $$
  and let $(x_j,y_j)$ be a maximum of $w_j$ relative to
  $\overline\dom\times\overline\dom$.  For $j$ sufficiently large,
  $(x_j,y_j)$ is an interior point. Moreover, $j|x_j-y_j|^4\to 0$ as
  $j\to\infty$, and $x_j,y_j\to \hat x$ and $\hat x$ is a maximum
  point for $u-v$ in $\dom$.

  Again we see that $v+\Psi_j(x_j,\cdot)$ has a local minimum at $y_j$, and thus
  \[
  -\Delta^N_{\p} (-\Psi_j(x_j,y_j))\ge \eps.
  \]
  If $p(y_j)\ge 2$, this implies that
  \[
  \begin{split}
    \eps \le&\, \Delta \Psi_j(x_j,y_j)+(p(y_j)-2)\lambda_{\max}(D^2  \Psi_j(x_j,y_j))\\
    =&\, j(n+2)\abs{x_j-y_j}^2+3j(p(y_j)-2)\abs{x_j-y_j}^2,
  \end{split}
  \]
  and if $1<p(y_j)<2$, then
  \[
  \begin{split}
    \eps \le&\, \Delta \Psi_j(x_j,y_j)+(p(y_j)-2)\lambda_{\min}(D^2  \Psi_j(x_j,y_j))\\
    =&\, j(n+2)\abs{x_j-y_j}^2+j(p(y_j)-2)\abs{x_j-y_j}^2.
  \end{split}
  \]
  In particular, in any case we must have $x_j\ne y_j$.

  Since $(x_j,y_j)$ is a local maximum point of $w_j(x,y)$, the
  maximum principle for semicontinuous functions implies that there
  exist symmetric $n\times n$ matrices $X_j,Y_j$ such that
  $$
  \begin{array}{r@{\;}l}
    (D_x\Psi_j(x_j,y_j),X_j)\in&\, \overline J^{2,+} u(x_j),\\[6pt]
    (-D_y\Psi_j(x_j,y_j),Y_j)\in&\, \overline J^{2,-} v(y_j),
  \end{array}
  $$
  and
\[
\begin{split}
\left(\begin{array}{@{}cc@{}}
 X_j&0\\
0&-Y_j\end{array}\right)
\le&\, D^2\Psi_j(x_j,y_j)+\frac 1j
\left[D^2\Psi_j(x_j,y_j)
\right]^2.
\end{split}
\]
We have
\[
\begin{split}
 D^2\Psi_j(x_j,y_j)=j\abs{z_j} \left(\begin{array}{@{}cc@{}}
I&-I\\-I&I\end{array}\right)+2j \left(\begin{array}{@{}cc@{}}
z_j\otimes z_j&-z_j\otimes z_j\\
-z_j\otimes z_j& z_j\otimes z_j\end{array}\right),
\end{split}
\]
where $z_j=x_j-y_j$, and thus
\begin{equation}
\label{matineq2}
\begin{split}
\left(\begin{array}{@{}cc@{}}
 X_j&0\\
0&-Y_j\end{array}\right)&\le j(|z_j|^{2}+2|z_j|^{4})\left(\begin{array}{@{}cc@{}}
I&-I\\-I&I\end{array}\right)\\
&\phantom{\le}+ 2j(1+8|z_j|^{2})\left(\begin{array}{@{}cc@{}}
z_j\otimes z_j&-z_j\otimes z_j\\
-z_j\otimes z_j& z_j\otimes z_j\end{array}\right).
\end{split}
\end{equation}
This implies
\begin{equation}\label{diffestimate}
  \begin{split}
     X_j\xi\cdot\xi - Y_j\zeta\cdot\zeta
\le\, j\left[3|z_j|^{2}+18|z_j|^{4}\right]|\xi-\zeta|^2
  \end{split}
\end{equation}
for all $\xi,\zeta\in\Rn$. In particular, $X_j\le Y_j$ and hence
$\tr(X_j)\le \tr(Y_j)$.

Since $u$ is a viscosity solution of $-\Delta^N_{\p} u=0$, $v$ satisfies
$-\Delta^N_{\p} v=\eps$, and $\eta_j=j\abs{x_j-y_j}^2(x_j-y_j)\ne 0$, we obtain
using \eqref{diffestimate} that
\[
\begin{split}
0<\eps\le&\, \tr(X_j-Y_j)+(p(x_j)-2)X_j\hat\eta_j\cdot\hat\eta_j
-(p(y_j)-2)Y_j\hat\eta_j\cdot\hat\eta_j\\
\le&\, (X_j-Y_j)\hat\eta_j\cdot\hat\eta_j +(p(x_j)-2)X_j\hat\eta_j\cdot\hat\eta_j
-(p(y_j)-2)Y_j\hat\eta_j\cdot\hat\eta_j\\
=&\, X_j(\sqrt{p(x_j)-1}\,\hat\eta_j)\cdot(\sqrt{p(x_j)-1}\,\hat\eta_j)
\\
&\hspace{10 em}
-Y_j(\sqrt{p(y_j)-1}\,\hat\eta_j)\cdot(\sqrt{p(y_j)-1}\,\hat\eta_j)\\
\le&\, Cj\abs{x_j-y_j}^2 \abs{\sqrt{p(x_j)-1}-\sqrt{p(y_j)-1}}^2\\
\le&\, Cj\abs{x_j-y_j}^2 \frac{\abs{p(x_j)-p(y_j)}^2}{(\sqrt{p(x_j)-1}+\sqrt{p(y_j)-1})^2},
\end{split}
\]
where $\hat\eta_j=\eta_j/\abs{\eta_j}$.
The right hand side tends to zero as $j\to\infty$ because $\p$ is assumed to be Lipschitz continuous,
$j\abs{x_j-y_j}^4\to 0$, and $p(\hat x)>1$. This contradiction completes the proof.
\end{proof}

\begin{remark} Unlike its counterpart for the divergence form $\p$-Laplacian \eqref{eq:div-equation},
Proposition \ref{norm_strictcomp} does not need the assumption
that either $u$ or $v$ is locally Lipschitz. However, if that is the case, then
it is enough to assume that $\p$ is just H\"older continuous with exponent
$\alpha>1/2$. Indeed, we have
\[
u(y_j)-v(y_j) \le u(x_j)-v(y_j)-\frac{j}4\abs{x_j-y_j}^4,
\]
and assuming that $u$ is Lipschitz, this gives
\[
\frac{j}4\abs{x_j-y_j}^4 \le u(x_j)-u(y_j) \le C \abs{x_j-y_j},
\]
and hence, after dividing by $\abs{x_j-y_j}^{1-\eps}$, we obtain $j\abs{x_j-y_j}^{3+\eps}\to 0$ as
$j\to\infty$ for any $\eps>0$. The last displayed inequality in the proof of Proposition
\ref{norm_strictcomp} gives
\[
0<\eps\le Cj\abs{x_j-y_j}^2\abs{p(x_j)-p(y_j)}^2 \le Cj\abs{x_j-y_j}^{2(1+\alpha)},
\]
and as $2(1+\alpha)>3$, we get a contradiction.
\end{remark}

\begin{remark}
Suppose that the following holds for an upper semicontinuous function
$u$: whenever $x_0\in\Om$ and $\vp \in C^2(\Om)$ are such that $u(x_0) = \vp(x_0)$, $u(x) < \vp(x)$ for $x \neq x_0$,
\emph{and} $D\vp(x_0)\ne 0$, we have
\[
-\Delta^N_{p(x)} \vp(x_0)\le 0.
\]
Then $u$ is a viscosity subsolution of $-\Delta^N_{p(x)} v(x)= 0$. In other words, the test-functions with vanishing gradient
at the point of touching can be completely ignored.

To see this, we argue by contradiction, and suppose that a function $u$ satisfying the condition above
is not a subsolution. Then there is $x_0\in\Om$ and
$\vp \in C^2(\Om)$ touching $u$ from above at $x_0$ such that $D\vp(x_0)= 0$, $D^2\vp(x_0)\ne 0$, and
\[
\begin{cases}
-\Delta \vp(x_0)-(p(x_0)-2)\lambda_{\min}(D^2\vp(x_0))> \eps,& \text{if $p(x_0)< 2$},\\
-\Delta \vp(x_0)-(p(x_0)-2)\lambda_{\max}(D^2\vp(x_0))>\eps,& \text{if $p(x_0)\ge 2$}
\end{cases}
\]
for some $\eps>0$. In particular, $\vp$ is a viscosity supersolution of $-\Delta^N_\p v(x)=\eps$ in some small ball
$B(x_0,\delta)$ and $u\le \vp$ on $\partial B(x_0,\delta)$. We can now run the proof of Proposition \ref{norm_strictcomp}
with $v$ replaced by $\vp$ and $\Om$ by $B(x_0,\delta)$, and contradict the assumption that $u-\vp$ has an interior
maximum at $x_0$. Therefore such a test-function cannot exist, and $u$ is a subsolution as claimed.
\end{remark}

\end{document}